\documentclass[letterpaper,12pt,reqno]{amsart}
\usepackage{graphicx,amsmath,amssymb, amsthm}

\usepackage{srcltx}
\usepackage[latin1]{inputenc}
\usepackage[T1]{fontenc}

\usepackage{fullpage}
\usepackage[colorlinks=true, pdfstartview=FitH, linkcolor=blue, citecolor=blue, urlcolor=blue]{hyperref}

\newtheorem{X}{X}[section]

\newtheorem{lemma}[X]{Lemma}
\newtheorem{proposition}[X]{Proposition}
\newtheorem{theorem}[X]{Theorem}

\theoremstyle{definition}
\newtheorem{remark}[X]{Remark}

\newcommand{\F}{{\mathcal F}}
\newcommand{\eps}{\varepsilon}

\newcommand{\be}{\overrightarrow{\beta}}

\newcommand{\ord}{{\rm ord}}
\newcommand{\one}{{\bf1}}

\newcommand{\ep}{\varepsilon}

\renewcommand{\be}{\begin{equation}}
\newcommand{\ee}{\end{equation}}
\newcommand\bea{\begin{eqnarray}}
\newcommand\eea{\end{eqnarray}}
\newcommand\bi{\begin{itemize}}
\newcommand\ei{\end{itemize}}
\newcommand\ben{\begin{enumerate}}
\newcommand\een{\end{enumerate}}
\newcommand\bc{\begin{center}}
\newcommand\ec{\end{center}}
\newcommand\ba{\begin{array}}
\newcommand\ea{\end{array}}

\title{A disproof of Hooley's conjecture}
\author{Daniel Fiorilli and Greg Martin}
\address{CNRS, Universit\'e Paris-Saclay, Laboratoire de math\'ematiques d'Orsay, 91405, Orsay, France.}
\email{daniel.fiorilli@universite-paris-saclay.fr}
\address{Department of Mathematics \\ University of British Columbia \\ Room
121, 1984 Mathematics Road \\ Canada V6T 1Z2}
\email{gerg@math.ubc.ca}
\date{\today}

\begin{document}

\begin{abstract}
Define $G(x;q)$ to be the variance of 
primes $p\le x$ in the arithmetic progressions modulo $q$, weighted by $\log p$.
Hooley conjectured that as soon as $q$ tends to infinity and $x\ge q$, we have the upper bound $G(x;q) \ll x \log q$.
In this paper we show that the upper bound does not hold in general, and that $G(x;q)$ can be asymptotically as large as $x (\log q+\log\log\log x)^2/4$.

\end{abstract}
\maketitle

\section{Introduction and statement of results}

For $x>q\geq 3$, we define the variance
$$ G(x;q):= \sum_{\substack{ a \bmod q \\ (a,q)=1}} \Big| \sum_{\substack{ p\leq x  \\ p\equiv a \bmod q}} \log p-\frac{x}{\phi(q)} \Big|^2,
$$
as well as the closely related (and perhaps slightly more natural)
$$
V_{\Lambda}(x;q):= \sum_{\substack{ a \bmod q \\ (a,q)=1}} \Big| \sum_{\substack{ n\leq x  \\ n\equiv a \bmod q}} \Lambda(n)-\frac{1}{\phi(q)}\sum_{\substack{ n\leq x  \\ (n,q)=1}} \Lambda(n) \Big|^2.$$ 
Since the pioneering work of Barban, Davenport and Halberstam~\cite{B,DH}, the study of this variance has seen a long line of developments, and continues to be an active research topic. To cite a few of the numerous papers written over the years, we mention the works~\cite{Mo,Ho1,FG,FG2,GV,V,Liu,P,HS}, as well as Hooley's series of 19 research papers and 2 survey papers (see for instance~\cite{Ho2,H5,Ho7,HoI,HoM}). We also have the recent works~\cite{MP,KR,BaFr,BW,S,CLLR}, which explore sparse averages over~$q$, as well as analogues for number fields and function fields.

Hooley conjectured~\cite[p.~217]{HoM},~\cite[equation~(2)]{H5} that as soon as~$q$ tends to infinity with~$x$, we have the upper bound
\begin{equation}
G(x;q)\ll x\log q.
\label{eq:HooleyOConj}
\end{equation}
Our main result is a disproof of this conjecture. 
\begin{theorem}
\label{theorem main}
There exists an infinite sequence of pairs $(q_j,x_j)$, with $q_j < x_j$ both tending to infinity, for which
\begin{equation}
\lim_{j\to\infty} \frac{G(x_j;q_j)}{x_j\log q_j} = \infty.
\label{equation disproof hooley}
\end{equation}
The same holds for $V_{\Lambda}(x_j;q_j)$ in place of $G(x_j;q_j)$.
\end{theorem}

One can ask whether the size of $G(x;q)$ and $V_{\Lambda}(x;q)$ in~\eqref{equation disproof hooley} 
can be made more explicit in terms of~$x$ and~$q$. Moreover, in light of the work of Keating and Rudnick~\cite[Section 2.2]{KR} and of the first author~\cite[Conjecture 1.1]{Fi}, one might be curious as to the corresponding size of~$q$ relative to~$x$. Finally, one can ask whether the density of the set of moduli~$q_j$ described in Theorem~\ref{th:main} can be quantified.

These questions are answered in the following more precise version of Theorem~\ref{theorem main}, in which we believe the range of~$q$ in terms of~$x$ is essentially best possible. To make this range explicit, we will consider functions $h\colon\mathbb R_{\geq 0} \rightarrow \mathbb R_{\geq 0}$ such that
\begin{equation} \label{nice function}
h(x) \text{ is increasing to infinity, and } h\big( e^{y^A} \big) \ll_{A,h} h(e^y) \text{ for every } A>1,
\end{equation}
the prototypical example of which is $h(x) = \max\{\log\log x,1\}$.

\begin{theorem}
\label{theorem main precise}
Let $h(x)$ be a function satisfying equation~\eqref{nice function}, and let $\ep>0$.
If
\begin{equation}
\frac{\ep \log\log x}{\log\log\log x} \leq  h(x) \le \log\log x
\label{equation condition h log2log3}
\end{equation}
when $x\geq e^3$, then
for a positive proportion of moduli $q$, there exist associated values $x_q$ such that $  q \asymp h(x_q)$ and
\begin{equation}
 G(x_q;q) \gg_{\ep}    x_q\log q \cdot  \frac{ \log\log x_q }{q}.
 \label{equation weaker omega result}
\end{equation}
In particular, when $\delta>0$ is sufficiently small, equation~\eqref{eq:HooleyOConj} cannot hold in any range of $q$ that satisfies $q < \delta \log\log x$.

If on the other hand the function
$$h(x)  \leq \eps   \frac{\log\log x}{\log\log\log x},
$$ 
for $x\ge e^3$,
 then for a positive proportion of moduli $q$, there exist associated values $x_q$ such that $  q \asymp h(x_q)$ and
$$G(x_q;q)  \geq  \big(\tfrac 14 -\ep \big) x_q \cdot ( \log q + \log\log\log x_q)^2 \gg x_q(\log q)^2.   $$
The same statements hold for $V_{\Lambda}(x_q;q)$ in place of $G(x_q,q)$.
\end{theorem}

\begin{remark}
Under GRH, the generalized Riemann hypothesis for Dirichlet $L$-functions, omega results similar to Theorem~\ref{theorem main precise} hold for \emph{all} moduli $q$. Indeed, we will show in Theorem~\ref{th:main} that for any function $h$ satisfying equation~\eqref{nice function} and $h(x) \leq \delta \log\log x$ with $\delta>0$ small enough, there exists a sequence $\{x_q\}_{q\geq 1}$ such that $  \phi(q) \asymp h(x_q)$ and having the property that $G(x_q;q)/(x_q\log q) $ tends to infinity as $q\rightarrow \infty$.

Our method also produces an omega result in a wider range, namely for 
$$ \frac{\log\log x}{\log\log\log x} \leq q \leq (\log x)^{\delta}$$
with~$\delta$ small enough. Indeed, $G(x;q)$ can be as large as  $(x \log\log x)( \log q)/ \phi(q)  $ (combine Theorem~\ref{th:main} with Proposition~\ref{prop:GRHfalse}). However, in this range sharper omega results will be obtained in~\cite{BrFi} for a weighted variant of $V_{\Lambda}(x;q)$ and all higher even moments. 
\end{remark}

One might wonder whether omega results analogous to Theorem~\ref{theorem main precise} also hold for fixed values of $q$, as $x\rightarrow \infty$. For instance, in the case $q=4$, 
$$ V_{\Lambda}(x;4) = \frac { |\psi(x,\chi_{-4})|^2}2 =\Omega\big(x (\log\log\log x)^2 \big) $$
by Littlewood's celebrated result~\cite{Li}. This result has been generalized by Davidoff to all real Dirichlet characters (private communication), and as a result, an application of Parseval's identity 
results in
\begin{equation}
 V_{\Lambda}(x;q) =\Omega\bigg( \frac { x (\log\log\log x)^2}{\phi(q)} \bigg), 
 \label{equation omega davidoff}
\end{equation}
which is comparable to~\cite[equation~(1.22)]{FG}.
In the next theorem we generalize the results of Littlewood and Davidoff to all characters modulo $q$. We also show that the large values of $|\psi(x,\chi)|$ with $\chi$ running over the Dirichlet characters modulo $q$ can be synchronized, and as a result we obtain an omega result sharper than~\eqref{equation omega davidoff}. 

\begin{theorem}
\label{theorem fixed q}
Let $q\geq 3$ be fixed. For any fixed nonprincipal character $\chi\bmod q$,
\begin{equation}
\Re(\theta(x,\chi)) = \Omega_- \big( x^{\frac 12} \log\log\log x \big).
 \label{equation omega result fixed chi}
\end{equation}
Moreover, 
\begin{equation}
 G(x;q) = \Omega \big( x (\log\log\log x)^2 \big).
 \label{equation omega result fixed q}
\end{equation}
The same oscillation resuts hold with $\psi(x,\chi)$ and $V_\Lambda(x;q)$ in place of $\theta(x,\chi)$ and $G(x;q)$, respectively.
\end{theorem}

\noindent We remark that the sequence of $x$-values implied by these oscillation results depends upon~$\chi$ or~$q$; the implied $\Omega$-constants, however, are absolute.

Our final theorem is an omega result on the average of $V(x;q)$ over $q$. Hooley conjectured the upper bound~\eqref{eq:HooleyOConj} based on his result~\cite[Theorem 1]{H5}, which states (assuming GRH) that uniformly for $q\leq x$, one has the upper bound
$$ \frac 1T\int_{\log 2}^{T} e^{-y}G(e^y;q) \,dy \ll \log q,  $$
that is, equation~\eqref{eq:HooleyOConj} holds on average over $\log x$. A natural question to ask here would be whether replacing this average with a more classical $q$-average would result in the same upper bound, that is, whether
\begin{equation}
\frac 1Q\sum_{q\le Q} G(x;q)\ll x\log Q.
\label{eq:HooleySumConj}
\end{equation}
As it turns out, this assertion is also false, as we show below.

\begin{theorem}
\label{theorem average over q}
Let $\ep>0$ be small enough, and let $Q\colon \mathbb R_{\geq 0} \rightarrow \mathbb N$ be a monotonic function with the property~\eqref{nice function} and such that $    Q(x) \leq \ep (\log\log x)^{\frac 12}/(\log\log\log x)^{\frac 12}$. Then we have the omega result 
\[
\frac 1{Q(x)}\sum_{Q(x) < q \le 2Q(x)} G(x;q) = \Omega\big(x(\log\log\log x)^2 \big),
\]
and the same statement holds for $V_{\Lambda}(x;q)$ in place of $G(x;q)$.
\end{theorem}

Let us briefly describe the tools used in the proofs of Theorems~\ref{theorem main precise},~\ref{theorem fixed q}, and~\ref{theorem average over q}. The first step, which is carried out in Section~\ref{section non GRH}, is to show that if GRH is false, then a stronger but ineffective omega result holds. Indeed, if $L(s,\chi)$ has a non-trivial zero $\rho_{\chi} = \Theta_\chi +i\gamma_\chi$ off the critical line, then by Landau's theorem $|\theta(x,\chi)-\one_{\chi=\chi_0}x|$ can be as large as $x^{\Theta_{\chi}-\ep}$. (Here $\one_{\chi=\chi_0}$ equals~$1$ if $\chi=\chi_0$ and~$0$ otherwise.) This works well for fixed moduli $q$ (as in Theorem~\ref{theorem fixed q}); however, one needs to modify this approach to have a result which is uniform in the range $q\leq x^{o(1)}$ (for Theorem~\ref{theorem main precise}). To achieve this, we apply Parseval's identity~\eqref{equation parseval} and positivity.

Let us use GRH$(\chi)$ to denote the generalized Riemann hypothesis for a specific Dirichlet $L$-function $L(s,\chi)$. If $q_e$ is the least modulus for which a character $\chi_e$ exists such that GRH$(\chi_e)$ is false,
then $\chi_e$ will induce a character modulo every multiple of $q_e$ whose associated Dirichlet $L$-function also violates GRH. As a result, we will deduce (see Proposition~\ref{prop:GRHfalse}) that $G(x;q)$ can be as large as $x^{2\Theta_{\chi_e}-\ep}/\phi(q)$;
since $ \Theta_{\chi_e}$ is independent of $q$ and $x$, this deduction will violate Hooley's conjecture~\eqref{eq:HooleyOConj} in the range $q\leq x^{o(1)}$ for a positive proportion of moduli~$q$. In other words, Hooley's conjecture in any range of the form $q\asymp x^{o(1)}$ is stronger than GRH; indeed, Hooley's conjecture in a range of the form $q\asymp x^\delta$ implies the zero-free strip $\Re(s) > \frac12+\frac\delta2$ for all Dirichlet $L$-functions modulo~$q$ (see Proposition~\ref{proposition Hooley implies GRH}).\

A difficulty arises in this approach when one is looking for a result which holds for many values of~$q$. Indeed, if one uses the oscillations of $ \theta(x,\chi_e)-\one_{\chi=\chi_0}x$ to create large values of $G(x;q)$, say on a sequence~$x_j$, then the condition $q\asymp h(x_q)$ will force~$q$ to be in a set which is possibly thin, since Landau's theorem alone does not give the rate of growth of~$x_j$. To circumvent this possible issue we applied a refined omega result of Kaczorowski and Pintz~\cite{KP}, which gives a rate of growth for $x_j$; however, their main theorem requires the assumption that $L(\sigma,\chi_e)\neq 0$ for $\frac 12\leq \sigma<1$. Fortunately, for our purposes it is sufficient to apply a weaker result, which as we will show can be proven unconditionally. 

In the second step, which is more intricate, we assume that GRH holds. Our general strategy in Section~\ref{section GRH hard} is to apply the explicit formula and to synchronize the summands using homogeneous Diophantine approximation. This goes back to Littlewood's original approach, modified as in~\cite[Theorem 15.11]{MV} and \cite[Lemma 2.4]{RubSar}. However, working uniformly in~$q$ poses significant new challenges. 
 In the classical proofs, 
 the quantity $x^{-\frac 12}(\psi(x)-x)$ is expressed as a sum over zeros of $\zeta(s)$, which under the Riemann hypothesis is approximated~by 
\begin{equation}
-\sum_{\rho} \frac{x^{i\gamma}}{\rho}  \approx -\sum_{\gamma} \frac{x^{i\gamma}}{i\gamma}= -2\sum_{\gamma >0 } \frac{\sin(\gamma \log x)}{\gamma}.
 \label{equation usual approximation}
\end{equation}  
If one applies the same approximation to $\psi(x,\chi)$, then one runs into the problem of potential real zeros $\rho_\chi=\tfrac 12$, that is, $\gamma_\chi=0$. Even assuming Chowla's conjecture $L(\tfrac 12,\chi)\neq 0$, we would still not be able to control the ordinate of the smallest zero uniformly in~$q$. 
Moreover, the approximation~\eqref{equation usual approximation} has an error term of size $x^{\frac 12}$, which translates to $x^{\frac 12} \log q$ in  the case of $\psi(x,\chi)$. This error term would need to be smaller than a main term $ \asymp  x^{\frac 12}\log(q\log\log x)$, which is not the case in the range $ (\log\log x)^{\delta} \leq q \leq \log\log x$.

In order to circumvent these issues, we keep the term on the left hand side of equation~\eqref{equation usual approximation} as it is, and bound the contribution of real zeros by applying a generalization of the $1$-level density result of Hughes and Rudnick~\cite{HR}, as well as a result of Selberg~\cite{Se} on multiplicities of zeros of $L(s,\chi)$ (see Lemma~\ref{lemma multiplicities} for both applications). We compute the average of $\psi(e^y,\chi)$ in suitable short intervals, which will be determined by an application of homogeneous Diophantine approximation, to synchronize the frequencies $   y \gamma_\chi  / 2\pi \bmod 1 $  simultaneously for all zeros $\rho_{\chi}=\tfrac 12+i\gamma_\chi$, with height at most $T$, of all $L(s,\chi)$ with $\chi  \bmod q$. Interestingly, in certain ranges, rather than synchronizing an unbounded number of frequencies for a single character, we synchronize a large enough but bounded number of frequencies for each character modulo~$q$. This step forces the value of $ \log x$ to be as large as $\exp \big( c\phi(q) T \log(qT) \big)$ for some constant $c>0$, which explains the range $q\leq \log\log x/\log\log\log x$ in the second part of Theorem~\ref{theorem main precise}.

We also localize the large values fairly precisely in Theorem~\ref{theorem main precise}, that is, we obtain two-sided bounds on $x$ in terms of $q$ and $T$. The key observation here is that we can exploit the almost-periodicity of $\psi(e^y,\chi)$ as in~\cite[Section 2.2]{RubSar} by finding \emph{many} values of $n$ in the Diophantine approximation step, which will force one of these values to be $\geq \exp\big( \tfrac c3 \phi(q) T \log(qT) \big)$.
 Once this is done, the last step is to estimate the resulting sums using the Riemann--von Mangoldt formula and an evaluation of the average log-conductor~\cite[Proposition 3.3]{FiMa}, which yields the second part of Theorem~\ref{theorem main precise}. 

In order to obtain the full range in Theorem~\ref{theorem main precise} (that is, $q\leq \ep \log\log x$), this approach needs to be further modified. Indeed, the fact that we are synchronizing the frequencies $   y \gamma_\chi  / 2\pi \bmod 1 $ for \emph{all} characters $\chi \bmod q$ forces $x$ to be as large as $\exp(q^{O(q)})$. To reduce this bound, we instead synchronize a \emph{subset} of characters modulo $q$, resulting in the  weaker omega result~\eqref{equation weaker omega result}, which is still strong enough to contradict Hooley's conjecture~\eqref{eq:HooleyOConj}. This approach introduces additional difficulties in the estimation  of the average of the log-conductor, which are overcome by applying the recent statistical results~\cite[Lemma 3.1]{BrFi} and~\cite[Lemma 3.2]{Fi} (see Lemmas~\ref{lemma subset with big conductors} and~\ref{lemma final expression for R} below).

In summary, in Section~\ref{section non GRH} we establish propositions that imply our main theorems when GRH is false; in Section~\ref{section GRH hard} we prove more delicate results that imply our main theorems when GRH is true. In particular, our main technical result is Theorem~\ref{th:main}, after which we deduce Theorems~\ref{theorem main}, \ref{theorem main precise}, \ref{theorem fixed q}, and~\ref{theorem average over q}.

\section{Hooley's conjecture and GRH}

\label{section non GRH}
The goal of this section is to show that Hooley's conjecture~\eqref{eq:HooleyOConj} in any range of the form $q\asymp x^{o(1)}$ is stronger than GRH; as a result, in the proof of Theorem~\ref{theorem main} (that is, when disproving Hooley's conjecture) we will be able to assume GRH in subsequent sections.
We will use the classical notation
$$  \psi(x;q,a):= \sum_{\substack{ n\leq x  \\ n\equiv a \bmod q}} \Lambda(n); \hspace{1cm} \theta(x;q,a):= \sum_{\substack{ p\leq x  \\ p\equiv a \bmod q}} \log p,$$
and for a Dirichlet character $\chi \bmod q$, 
$$ \psi(x,\chi):= \sum_{n\leq x} \chi(n) \Lambda(n); \hspace{1cm} \theta(x,\chi):= \sum_{p\leq x} \chi(p) \log p.$$
We record the Parseval identities
\begin{equation}
 V_{\Lambda}(x;q)=\frac 1{\phi(q)} \sum_{\substack{ \chi \bmod q \\ \chi \neq \chi_0 }} |\psi(x,\chi)|^2 ; \hspace{1cm} G(x;q)= \frac 1{\phi(q)} \sum_{\chi \bmod q} |\theta(x,\chi)-\one_{\chi = \chi_0} x|^2.
 \label{equation parseval}
\end{equation}

Our first step is to see that if $\chi \bmod q$ is a character for which $L(s,\chi)$ does not satisfy the Riemann Hypothesis, then $|\psi(x,\chi)|$ and $|\theta(x,\chi)-\one_{\chi=\chi_0}x|$ have large values. This will follow from a result of Kaczorowski and Pintz~\cite{KP} which we will adapt in order to obtain an unconditional statement. 
 One should keep in mind that this will give  no information about uniformity in $q$. We also mention that for $\chi=\chi_0$, we have the more precise results of Pintz~\cite{Pi} and Schlage--Puchta~\cite{SP}.

\begin{lemma}
\label{lemma:complexoscillations}
Fix $\eps>0$, let $q\geq 1$, and let $\chi$ be a character ${}\bmod q$. Define $\Theta_{\chi}\geq \tfrac 12$ to be the supremum of the real parts of the zeros of $L(s,\chi)$. Then, for every large enough $X$ (in terms of $\chi$ and $\eps$), there exists $x\in [X^{1-\eps},X]$ such that  
$$ \Re(\psi(x,\chi))-\one_{\chi=\chi_0}x <- x^{\Theta_{\chi}-\ep}.$$
\end{lemma}

\begin{proof}
Suppose first that $L(\Theta_{\chi},\chi)\ne0$.
Then the claim follows from setting $f(x)=\Re(\psi(x,\chi)-\one_{\chi=\chi_0}\one_{x\geq 1}x)$ and applying~\cite[Theorem 1]{KP}. Indeed,
$$ \int_{0}^\infty f(x) x^{-s-1} d x=  -\frac 1{2s} \Big( \frac{L'(s,\chi)}{L(s,\chi)}+\frac{L'(s,\overline{\chi})}{L(s,\overline{\chi})}\Big)-\frac{ \one_{\chi=\chi_0}}{s-1}, $$
which is regular in the half plane $\Re(s)> \Theta_\chi$ but not in any half plane of the form $\Re(s)> \Theta_\chi-\eps$. Here we used the fact that the residues of $L'(s,\chi)/L(s,\chi)$ are nonnegative, since $L(s,\chi)-\one_{\chi=\chi_0}/(s-1)$ is entire (in other words, the poles of $L'(s,\chi)/L(s,\chi)$ and $L'(s,\overline{\chi})/L(s,\overline{\chi})$ cannot cancel each other).
 
Suppose otherwise that $L(\Theta_{\chi},\chi)=0$, and thus $\chi\neq \chi_0$. The explicit formula~\cite[Theorems~12.5 and~12.10]{MV} implies that for $T\geq 1$,
$$  \psi(x,\chi)
= -\sum_{\substack{\rho_{\chi} \\ |\Im(\rho_\chi)|\leq T}} \frac{x^{\rho_\chi}}{\rho_{\chi}} +O\Big( \log(qx)+\frac{ x (\log(qxT))^2}{T}\Big).$$
We deduce that for any $0<T_1<T_2 $,
\begin{multline*} \frac 1{T_2-T_1} \int_{T_1}^{T_2}  e^{-\Theta_\chi t}  \psi(e^t,\chi) \,dt = -\sum_{\rho_\chi\neq \Theta_\chi} \frac{e^{T_2 (\rho_\chi-\Theta_\chi)} -e^{T_1 (\rho_\chi-\Theta_\chi)  }}{\rho_{\chi}(\rho_{\chi}-\Theta_\chi)(T_2-T_1)} - \frac{\ord_{s=\Theta_\chi} L(s,\chi)}{\Theta_\chi} \\ +O\Big( e^{-\Theta_\chi T_1} \log(qT_2)\Big).
\end{multline*}
Taking $T_1=(1-\eps)\log X$ and $T_2=\log X$, and noting that the infinite sum over zeros in the last equation converges absolutely, we deduce that for large enough $X$ there exists $x\in [X^{1-\eps},X]$ for which 
$$x^{-\Theta_\chi }\Re(\psi(x,\chi)) < - \frac{\ord_{s=\Theta_\chi} L(s,\chi)}{2\Theta_\chi}. $$
The claim follows.
\end{proof}

With this omega result in hand, we will deduce that in certain ranges, the upper bound~\eqref{eq:HooleyOConj} is stronger than GRH. This is made precise in the following proposition. The goal here is to overcome the uniformity problems caused by the fact that $\Theta_\chi$ depends on~$\chi$ in Lemma~\ref{lemma:complexoscillations}. This will be done by noticing that for many moduli $q$, characters of small conductor occur in the sums in equation~\eqref{equation parseval}. 

\begin{proposition}
\label{proposition Hooley implies GRH}
Fix $0<\delta<1$, and assume that equation~\eqref{eq:HooleyOConj} holds in the range $x^{\delta} \leq q \leq 2x^{\delta}$. Then every Dirichlet $L$-function is nonvanishing in the half-plane $\Re(s)> \tfrac 12+ \tfrac {\delta}2$. If one replaces $G(x;q)$ with $V_{\Lambda}(x;q)$ in equation~\eqref{eq:HooleyOConj}, then the same half-plane is zero-free for all Dirichlet $L$-functions corresponding to nonprincipal characters.
\end{proposition}

\begin{proof}
We prove the contrapositive. Suppose that there exists a primitive Dirichlet $L$-function $L(s,\chi_e)$ of conductor $q_e\ge 1$ that has a zero with real part $\beta_e > \frac12+\frac\delta2$. By Lemma~\ref{lemma:complexoscillations}, for every $0<\ep < \beta_e$ there exists an increasing sequence $\{x_j\}_{j\geq 1}$ tending to infinity such that
\[
|\theta(x_j,\chi_e)-\one_{\chi=\chi_0}x_j| \geq x_j^{\beta_e-\ep}.
\]
We may assume that each $x_j > q_e^{1/\delta}$, so that any interval of length $x_j^\delta$ contains a multiple of~$q_e$.
For each $j\ge1$, choose an integer $x_j^\delta \le q_j \le 2x_j^\delta$ that is a multiple of~$q_e$, and let $\chi_j$ be the character${}\bmod q_j$ induced by~$\chi_e$. Note that
\begin{equation}  \label{GRH false step 1}
|\theta(x_j,\chi_e) - \theta(x_j,\chi_j)| \le \log q_j \ll_\delta \log x_j,
\end{equation}
and hence for $j$ large enough in terms of $\chi_e$ and $\eps$,
\begin{equation}  \label{GRH false step 2}
|\theta(x_j,\chi_j)-\one_{\chi=\chi_0} x_j| \geq x_j^{\beta_e-\ep}
\end{equation}
as well. Consequently, when $j$ is large enough we have that
\begin{align}\notag  
G(x_j;q_j) &= \frac1{\phi(q_j)} \sum_{\substack{\chi\bmod{q_j}}} |\theta(x_j,\chi)-\one_{\chi=\chi_0} x_j|^2 \\
&\ge \frac{|\theta(x_j,\chi_j)-\one_{\chi=\chi_0} x_j|^2}{q_j} \ge \frac{x_j^{2(\beta_e-\ep)}}{2x_j^\delta} = \frac{x_j^{2\beta_e-\delta-2\ep}}2,\label{GRH false step 3}
\end{align}
so that
\begin{equation}  \label{GRH false step 4}
\frac{G(x_j;q_j)}{x_j\log q_j} \gg \frac{x_j^{2\beta_e-1-\delta-2\ep}}{\log q_j} \gg \frac{x_j^{2\beta_e-1-\delta-2\ep}}{\delta \log x_j}.
\end{equation}
By assumption, $2\beta_e-1-\delta>0$, and so the exponent $2\beta_e-1-\delta-2\ep$ is positive as long as $\ep$ is chosen small enough. Therefore
\[
\lim_{j\to\infty} \frac{G(x_j;q_j)}{x_j\log q_j} = \infty,
\]
contradicting equation~\eqref{eq:HooleyOConj}. The proof is identical for $V_{\Lambda}(x;q)$.
\end{proof}

We now adapt the arguments in the proof of Proposition~\ref{proposition Hooley implies GRH} to prove a proposition that is more suitable for the proof of Theorem~\ref{theorem main}. 

\begin{proposition}
Assume that GRH is false. Then there exists an absolute constant $\delta>0$ with the following property. Let $h(x)$ be an increasing function tending to infinity such that $h(x) = o(x^\delta)$ as $x\rightarrow \infty$. For a positive proportion of moduli~$q$, there exist associated values~$x_q$ such that 
$h(x_q^{1-\delta})\leq q\leq h(x_q)$ and
$$ G(x_q;q) \geq x_q^{1+\delta}.$$ 
If GRH$(\chi)$ is false for some nonprincipal character~$\chi$, then the same lower bound holds with $V_{\Lambda}(x_q;q)$ in place of $G(x_q;q)$.
\label{prop:GRHfalse}
\end{proposition}

\begin{proof}
Fix a modulus $q_e\geq 1$ for which there exists an associated primitive character $\chi_e$ such that $L(s,\chi_e)$ has a zero with real 
part $\beta_{e} > \tfrac 12$. Fix a positive number $\ep < \beta_e-\frac12$, and choose a positive number $\delta < \beta_e-\frac12-\ep \leq \frac 12$. Now let~$q$ be any large enough multiple of~$q_e$; the set of such moduli~$q$ has positive (though ineffective) density in~$\mathbb N$. 
By Lemma~\ref{lemma:complexoscillations}, there exists $x_q\in [h^{-1}(q),h^{-1}(q)^{\frac 1{1-\delta}}]$ such that 
$$ | \theta(x_q,\chi_e)| > x_q^{\beta_{e}-\ep}.$$
Note that this implies that $q	\in [h(x_q^{1-\delta}),h(x_q)]$. Denote by $\chi_q$ the character${}\bmod{q}$ induced by~$\chi_e$. 
Then the calculations in equations~\eqref{GRH false step 1} through~\eqref{GRH false step 3} apply exactly to this situation; we conclude that
\begin{equation*}
G(x_q;q) \gg x_q^{2\beta_e-\delta-2\ep},
\end{equation*}
and the right-hand side is eventually larger than $x_j^{1+\delta}$ by our choice of~$\delta$, establishing the asserted lower bound. The proof for $V_{\Lambda}(x;q)$ is identical as long as $\chi_e$ is nonprincipal.
\end{proof}

To end this section, we further adapt Proposition~\ref{proposition Hooley implies GRH} with the aim of proving Theorem~\ref{theorem fixed q}. This situation is much easier since there are no uniformity issues (that is,~$q$ is fixed). 

\begin{proposition}
\label{proposition GRH false single chi}
Fix $q\geq 1$, and assume that there exists a character $\chi_e \bmod q$ such that GRH$(\chi_e)$ is false. Then, there exists a sequence $\{x_i\}_{i\geq 1}$, depending on $q$, such that for each $\eps>0$,
$$ G(x_i;q) \geq \frac 1{\phi(q)} |\theta(x_i,\chi_e)-\one_{\chi_e=\chi_0}x_i|^2 \gg_{\eps,q} x_i^{2\Theta_{\chi_e}-\eps}.$$
Here, $\Theta_{\chi_e}$ is the supremum of real parts of zeros of $L(s,\chi_e)$ with $\chi\bmod q$. Similarly for $V_{\Lambda}(x;q)$ and $\psi(x,\chi)$. 
\end{proposition}
\begin{proof}
This follows at once from equation~\eqref{equation parseval} and Lemma~\ref{lemma:complexoscillations}.
\end{proof}

\section{Explicit formulas and homogeneous Diophantine approximation}
\label{section GRH hard}

The goal of this section is to show that GRH implies Theorem~\ref{theorem main precise}. Our goal will be to synchronize the arguments of the summands in the explicit formula, but only for a subset~$\F_q$ of the set~$\mathcal X_q$ of characters modulo~$q$.

Throughout,~$\gamma_{\chi}$ denotes the imaginary part of a nontrivial zero of $L(s,\chi)$, and~${q_\chi}$ denotes the conductor of~$\chi$. We let $\|t\|$ denote the distance from~$t$ to the nearest integer, and we use the shorthand $\log_2 t=\log\log t$ and $\log_3 t=\log\log\log t$.

\begin{lemma}
\label{lemma multiplicities}
Let $q\geq 3$ be an integer. If $\chi\bmod q$ is a nonprincipal character, then assuming GRH$(\chi)$ we have the bound
$$  \ord_{s=\frac 12} L(s,\chi) \ll \frac{\log q}{\log_2 q}.$$
Moreover, if GRH$(\chi)$ is true for all nonprincipal $\chi \bmod q$, then
\begin{equation*}
 \sum_{\chi \bmod q} \ord_{s=\frac 12} L(s,\chi) \leq \bigg( \frac 12+o_{q\rightarrow \infty}(1) \bigg) \phi(q) .
\end{equation*}

\end{lemma}

\begin{proof}
The first bound is well known (see~\cite{Se} or~\cite[Proposition 5.21]{IwKo}). As for the second, it follows from~\cite[Theorem 2.1]{FiMi} (which generalizes  \cite{HR} to composite moduli). Indeed, letting $f(x) := (\sin(2\pi x)/2\pi x)^2$, we have
\begin{equation} \label{sin2 fourier}
\widehat f(x)= 
\begin{cases}
\frac 12-\frac{|x|}4, & \text{ if } |x|\leq 2, \\
0, &\text{ otherwise.}
\end{cases}
\end{equation}
Since $\zeta(\tfrac 12)\neq 0$ and we are assuming that all non-trivial zeros of $L(s,\chi)$ with $\chi\neq \chi_0$ have real part $\tfrac 12$, it follows that 
$$
 \sum_{\chi \bmod q} \ord_{s=\frac 12} L(s,\chi) \leq \sum_{\substack{\chi \bmod q \\ \chi \neq \chi_0}} f\Big( \frac {\log q }{2\pi } \gamma_{\chi}\Big),
 $$
which by~\cite[Theorem 2.1]{FiMi} is $\lesssim \phi(q)/2$.
\end{proof}

We will also need to control the conductors of the characters in~$\F_q$. We define $\Phi_q=\#\F_q$. We will require $\F_q$ to have the property that
\begin{equation}
\chi \in \F_q \text{ if and only if } \overline{\chi} \in\F_q.
\label{equation property 1 F_q}
\end{equation}

\begin{lemma}
\label{lemma subset with big conductors}
Let $w(q)>0$ be any function tending to zero as $q\rightarrow \infty$. For each $q\geq 3$, there exists a subset $\F_q \subset \mathcal X_q$ of the set of characters modulo $q$ of cardinality 
\begin{equation} \label{most of em}
\Phi_q = \# \F_q \geq \phi(q) \big( 1 -O(w(q)^2) \big),
\end{equation}
having the property~\eqref{equation property 1 F_q},
such that $\log q_\chi =  \log q +O\big( w(q)^{-1} \log_2 q \big)$ for each character $\chi \in \F_q$.
\end{lemma}
 
\begin{proof}
By~\cite[Lemma 3.1]{BrFi}, we have the estimate 
\begin{equation}
\frac{1}{\phi(q)} \sum_{\substack{\chi \bmod q }}(\log q_{\chi} -\log q)^2    \ll (\log_2 q)^2 .  \label{equation BrFi}
\end{equation}
Combined with~\cite[Proposition 3.3]{FiMa} and Chebyshev's inequality, this yields
$$\frac 1{\phi(q)} \# \{ \chi \bmod q \colon |\log q_\chi -\log q| >  w(q)^{-1}\log_2 q \} \ll w(q)^2. $$
Since $q_{\overline\chi}=q_\chi$, the characters not included in the above set have the property~\eqref{equation property 1 F_q}, and the desired estimate~\eqref{most of em} follows.
\end{proof}

We are now ready to estimate the logarithmic averages of $\psi(x,\chi)$ and $\theta(x,\chi)$ over a short interval. By carefully choosing this interval to synchronize the frequencies in the explicit formula, we will ultimately create large values of $G(x;q)$ and $V_{\Lambda}(x;q)$.
The Riemann--von Mangoldt formula
\begin{equation}
 N(T,\chi) := \# \{ \rho_\chi \colon |\Im(\rho_\chi)| \leq T\}  = \frac T{\pi} \log \Big( \frac{{q_\chi}T}{2\pi e} \Big) +O(\log(qT))
\label{equation von mangoldt}
\end{equation}
for $T\ge2$ will be central in our analysis. The next lemma records some estimates that follow easily from this asymptotic formula and partial summation.

\begin{lemma}
\label{N(T) partial summation lemma}
For any real parameters, $0<\delta < 1$ and $T\ge\delta^{-1}$,
\begin{align*}
\sum_{|\gamma_\chi| > T } \frac{e^{iy\gamma_\chi }}{\rho_\chi^2}\Big( \frac{i \sin(\delta \gamma_\chi)}{\delta} +\frac{\cos(\delta \gamma_\chi)}2 \Big) &\ll \frac1\delta \sum_{|\gamma_\chi| > T } \frac1{\gamma_\chi^2} \ll \frac{\log qT}{\delta T}; \\
\max\bigg\{ \sum_{0 \leq  \gamma_\chi  \leq T}  \frac{\sin^2(\delta \gamma_\chi)}{\delta|\rho_\chi|^4},  \sum_{0 \leq  \gamma_\chi  \leq T}  \frac{\gamma_\chi|\sin(\delta \gamma_\chi)\cos(\delta \gamma_\chi)|}{|\rho_\chi|^4} \bigg\} &\ll \delta \sum_{0 \leq  \gamma_\chi  \leq T} \frac 1{\gamma_\chi^2} \ll \delta \log q.
\end{align*}
\end{lemma}

The next lemma is a careful evaluation of the averages of $\psi(e^t,\chi)-\one_{\chi=\chi_0} e^t$ and $\theta(e^t,\chi)-\one_{\chi=\chi_0} e^t$ over a short interval. The specific interval will be chosen later using homegeneous Diophantine approximation, and will contain a large value of those functions.
\begin{lemma}
\label{lemma explicit formula precise}
Let $q\geq 1$ be an integer, and let~$\F_q$ be a set of characters modulo $q$ with the property~\eqref{equation property 1 F_q} such that GRH$(\chi)$ is true for all $\chi\in\F_q$.
Moreover, let $0<\delta < 1$, $T\ge\delta^{-1}$, and $M\geq 1$ be real parameters, and define
\begin{align*}
R_{\delta}(y) &:= \frac 1{2\delta} \int_{y-\delta}^{y+\delta} \sum_{\chi \in \F_q} (\psi(e^t,\chi)-\one_{\chi=\chi_0} e^t) \,dt; \\
S_{\delta}(y) &:= \frac 1{2\delta} \int_{y-\delta}^{y+\delta} \sum_{\chi \in \F_q} (\theta(e^t,\chi)- \one_{\chi=\chi_0} e^t)\,dt.
\end{align*}
\begin{enumerate}
\item For all $y\geq 0$,
\begin{equation}
 S_{\delta}(y) = R_{\delta}(y)+O(e^{\frac y2} \Phi_q).
 \label{equation S and R}
\end{equation} 
\item If $n\in \mathbb N$ satisfies $\| n \gamma_{\chi} \delta /2\pi \|< M^{-1} $ for each $ 0\leq \gamma_{\chi} \leq T$ with $\chi \in \F_q$, then $y=(n+1) \delta$ has the property that
\begin{multline}
\label{equation lemma explicit} 
 R_{\delta}(y)= - e^{\frac y2}\sum_{\chi \in \F_q}\sum_{0 \leq \gamma_\chi \leq  T} \Big( \frac{2 \gamma_\chi\sin(\delta \gamma_\chi)(\gamma_\chi \sin(\delta \gamma_\chi)+\cos(\delta \gamma_\chi))}{\delta|\rho_\chi|^4}   +\Big(\frac 14-\gamma_\chi^2 \Big)\frac{\cos^2(\delta \gamma_\chi)}{|\rho_\chi|^4}  \Big) 
\\ + O\Big(e^{\frac y2}\Phi_q \Big(\frac{\log(qT)}{\delta T}+ \min\bigg( \frac{\phi(q)}{\Phi_q} , \frac{\log q}{\log_2 q} \bigg) + \frac{y\log q}{e^{\frac y2}}+ \frac{\log(qT)\log T}M  + \delta \log q\Big) \Big).
\end{multline}
\end{enumerate}
\end{lemma}

\begin{proof}
We begin by noting that part~(a) is a direct consequence of the bound 
$$|\psi(x,\chi)-\theta(x,\chi)| \leq \sum_{\substack{ p^k \leq x \\ k\geq 2}} \log p \ll x^{\frac 12}.$$

As for part~(b), we can transfer the question to primitive characters by noting that if $\chi^*$ denotes the primitive character inducing $\chi$, then
$$\sum_{\chi \in \F_q}\psi(e^t,\chi)- \sum_{\chi \in \F_q}\psi(e^t,\chi^*) \ll  \Phi_q t\log q .$$
Moreover, $L(s,\chi)$ and $L(s,\chi^*)$ have the same zeros on the critical line. 
Hence, the explicit formula \cite[Theorems 12.5 and 12.10]{MV} gives that for $y\geq 1$ and $S\geq 1$, 	
\begin{align*}
 R_{\delta}(y) &=  -\frac 1{2\delta}\int_{y-\delta}^{y+\delta}   \sum_{\chi \in \F_q}\sum_{\substack{\rho_\chi  \\ |\Im(\rho_\chi)| \leq S }} \frac{e^{t\rho_\chi }}{\rho_\chi} \,dt + O\Big(\Phi_q y\log q + \frac{\Phi_qe^y (\log(qe^yS))^2}{S}\Big) \\
& = - \frac 1{2\delta}\sum_{\chi \in \F_q}\sum_{\substack{\rho_\chi  \\ |\Im(\rho_\chi)| \leq S }} \frac{e^{y\rho_\chi}}{\rho_\chi^2}(e^{\delta \rho_\chi}-e^{-\delta \rho_\chi})+ O\Big(\Phi_q y\log q + \frac{\Phi_qe^y (\log(qe^yS))^2}{S}\Big) \\
& = - \frac 1{2\delta}\sum_{\chi \in \F_q}\sum_{\substack{\rho_\chi }} \frac{e^{y\rho_\chi}}{\rho_\chi^2}(e^{\delta \rho_\chi}-e^{-\delta \rho_\chi})+ O(\Phi_q y\log q )
\end{align*}
after taking $S\to\infty$. Using $e^{\pm\delta \rho_\chi} = e^{\pm i\delta \gamma_\chi }(1\pm\tfrac \delta 2+O(\delta^2))$ and truncating the infinite sum over zeros using Lemma~\ref{N(T) partial summation lemma}, we deduce the estimate
\begin{align}
R_{\delta}(y) &=  - e^{\frac y2}\sum_{\chi \in \F_q}\sum_{\rho_\chi } \frac{e^{iy\gamma_\chi }}{\rho_\chi^2}\Big( \frac{i \sin(\delta \gamma_\chi)}{\delta} +\frac{\cos(\delta \gamma_\chi)}2 \Big)+ O(\Phi_q y\log q+\Phi_q\delta e^{\frac y2}\log q ) \notag \\
\label{equation conjugates grouped}
& =  - e^{\frac y2}\sum_{\chi \in \F_q}\sum_{0 \leq  \gamma_\chi  \leq T} \Big( \frac{i \sin(\delta \gamma_\chi)}{\delta}\Big( \frac{e^{iy\gamma_\chi }}{\rho_\chi^2} - \frac{e^{-iy \gamma_\chi}}{\overline{\rho_\chi}^2}\Big) +\frac{\cos(\delta \gamma_\chi)}2\Big( \frac{e^{iy\gamma_\chi }}{\rho_\chi^2} + \frac{e^{-iy\gamma_\chi }}{\overline{\rho_\chi}^2}\Big) \Big) \\&\hspace{1cm}+ O\Big(e^{\frac y2}\frac{\Phi_q\log(qT)}{\delta T}+e^{\frac y2} \min\Big(\phi(q), \Phi_q \frac{\log q}{\log_2 q}\Big) +\Phi_q y\log q+\Phi_q\delta e^{\frac y2}\log q \Big). \notag
\end{align}
where we have grouped conjugate zeros together, using Lemma~\ref{lemma multiplicities} to bound the contribution from possible zeros at $s=\frac12$ for which this grouping is erroneous.

We now apply the hypothesis that $y=(n+1) \delta$ with $\| n \gamma_\chi \delta /2\pi \|< M^{-1} $ for each $0 \leq \gamma_{\chi} \leq T$ and $\chi \in \F_q$. It follows that $e^{\pm i y\gamma_\chi } = e^{\pm i \delta\gamma_\chi} (1+O(M^{-1})) $, and thus the main term~\eqref{equation conjugates grouped} equals

\begin{align*}
&- e^{\frac y2}\sum_{\chi\in\F_q}\sum_{0 \leq  \gamma_\chi  \leq T} \Big( \frac{i \sin(\delta \gamma_\chi)}{\delta}\Big( \frac{e^{i \delta\gamma_\chi}}{\rho_\chi^2} - \frac{e^{-i \delta\gamma_\chi}}{\overline{\rho_\chi}^2}\Big) +\frac{\cos(\delta \gamma_\chi)}2\Big( \frac{e^{i \delta\gamma_\chi}}{\rho_\chi^2} + \frac{e^{-i \delta\gamma_\chi}}{\overline{\rho_\chi}^2}\Big) \Big) \\ &\hspace{10cm} +O\Big( \frac{e^{\frac y2}}M \sum_{\chi\in\F_q}\sum_{0\leq \gamma_\chi \leq T} \frac 1{|\rho_\chi|} \Big) \\
&=-  e^{\frac y2}\sum_{\chi\in\F_q}\sum_{0 \leq  \gamma_\chi  \leq T} \bigg\{ \frac{-2 \sin(\delta \gamma_\chi)}{\delta|\rho_\chi|^4}\Big( \Big(\frac 14-\gamma_\chi^2 \Big) \sin( \delta\gamma_\chi)-\gamma_\chi\cos(\delta\gamma_\chi ) \Big) \\
&\hspace{1cm}+\frac{\cos(\delta \gamma_\chi)}{|\rho_\chi|^4}\Big( \Big(\frac 14-\gamma_\chi^2 \Big)\cos( \delta\gamma_\chi) + \gamma_\chi \sin( \delta\gamma_\chi) \Big) \bigg\} +O\Big( \Phi_q\frac{e^{\frac y2}}M \log(qT) \log T \Big) \\
&= - e^{\frac y2}\sum_{\chi\in\F_q}\sum_{0 \leq  \gamma_\chi  \leq T} \Big( \frac{2 \gamma_\chi\sin(\delta \gamma_\chi)(\gamma_\chi\sin(\delta \gamma_\chi)+\cos(\delta \gamma_\chi))}{\delta|\rho_\chi|^4}   +\Big(\frac 14-\gamma_\chi^2 \Big)\frac{\cos^2(\delta \gamma_\chi)}{|\rho_\chi|^4}  \Big) \\ &\hspace{1cm}+O\Big( \Phi_q\frac{e^{\frac y2}}M \log(qT) \log T + \Phi_q\delta e^{\frac y2} \log q \Big)
\end{align*} 
by Lemma~\ref{N(T) partial summation lemma}.
\end{proof}

Now that we have expressed averages of $\theta(x,\chi)$ and $\psi(x,\chi)$ in suitable short intervals in terms of sums over zeros, our strategy is to estimate the sums in equation~\eqref{equation lemma explicit} using the Riemann--von Mangoldt formula~\eqref{equation von mangoldt}. 

\begin{lemma}
\label{lemma sum over zeros}
Let $q\geq 1$ be an integer, and let~$\F_q$ be a set of characters modulo $q$ with the property~\eqref{equation property 1 F_q} such that GRH$(\chi)$ is true for all $\chi\in\F_q$.
For any $0<\delta\le e^{-1}$ and $T \ge e\delta^{-1}$,
\begin{multline}
\delta\sum_{\chi \in \F_q}\sum_{0\leq \gamma_\chi \leq T} \frac{2\gamma_\chi^4}{|\rho_\chi|^4}\frac{\sin^2(\delta \gamma_\chi)}{(\delta\gamma_\chi)^2} =\frac{\Phi_q}2 \log(q\delta^{-1})
\\
 +O\Big(\delta \Phi_q \log(qT)\log T+\delta^{\frac 12 } \Phi_q \log (q\delta^{-1})+(E_q+\Phi_q) \log(T\delta) + \frac{\Phi_q \log(qT)}{\delta T}  \Big),
\label{equation sum over zeros to evaluate}
\end{multline}
where
\begin{equation}
 E_q := \sum_{\chi \in \F_q} (\log q-\log {q_\chi}).
 \label{equation definition E_q}
\end{equation}

\end{lemma}

\begin{proof}
We will see that the main contribution in the sum on the left hand side of~\eqref{equation sum over zeros to evaluate} comes from zeros $\gamma_{\chi}$ of intermediate size. We can therefore discard low-lying zeros as follows:
\begin{align*}
\delta\sum_{\chi\in \F_q}\sum_{0\leq \gamma_\chi \leq T} \frac{2\gamma_\chi^4}{|\rho_\chi|^4} & \frac{\sin^2(\delta \gamma_\chi)}{(\delta\gamma_\chi)^2} = \delta\sum_{\chi\in \F_q}\sum_{\delta^{-\frac 12}< \gamma_\chi \leq T} \frac{2\gamma_\chi^4}{|\rho_\chi|^4}\frac{\sin^2(\delta \gamma_\chi)}{(\delta\gamma_\chi)^2} +O\bigg( \delta \sum_{\chi\in \F_q} N(\delta^{-\frac12},\chi) \bigg) \\
&= \delta\sum_{\chi\in \F_q}\sum_{\delta^{-\frac 12}< \gamma_\chi \leq T} \bigg( 2 + O\bigg( \frac1{\gamma_\chi^2} \bigg) \bigg) \frac{\sin^2(\delta \gamma_\chi)}{(\delta\gamma_\chi)^2} +O(\Phi_q \delta^{\frac 12} \log (q\delta^{-\frac12}) ) \\
&=2\delta\sum_{\chi\in \F_q}\sum_{ \delta^{-\frac 12} < \gamma_\chi \leq T} \frac{\sin^2(\delta \gamma_\chi)}{(\delta\gamma_\chi)^2}+O(\Phi_q\delta^{ \frac 32}\log(q\delta^{-\frac12})+\Phi_q\delta^{\frac 12} \log (q\delta^{-1}) ),
\end{align*} 
and the first error term is smaller than the second.
Define
\begin{align*}
N(t,\F_q) := \sum_{\chi\in\F_q} N(t,\chi) &= \sum_{\chi\in\F_q} \bigg( \frac t{\pi} \log \Big( \frac{{q_\chi}t}{2\pi e} \Big) +O(\log(qt)) \bigg) \\
&= \sum_{\chi\in\F_q} \frac t{\pi} \log(qt) - \sum_{\chi\in\F_q} \frac t{\pi} (\log q - \log q_\chi ) + O\bigg( \sum_{\chi\in\F_q} \big( t+\log(qt) \big) \bigg) \\
&= \Phi_q \frac t{\pi} \log(qt) + O\big( E_q t + \Phi_q \big( t+\log(qt) \big) \big)
\end{align*}
by the asymptotic formula~\eqref{equation von mangoldt}.
Notice that counting only zeros above the real axis would yield $\frac12N(t,\F_q)$ in place of $N(t,\F_q)$ thanks to the property~\eqref{equation property 1 F_q} and the functional equation for Dirichlet $L$-functions.
We may now compute
\begin{align}
2\delta & \sum_{\chi\in \F_q}\sum_{ \delta^{-\frac 12} < \gamma_\chi \leq T} \frac{\sin^2(\delta \gamma_\chi)}{(\delta\gamma_\chi)^2} = 2\delta\int_{\delta^{-\frac 12}}^T \frac{\sin^2(\delta t)}{(\delta t)^2} \,d\big( \tfrac12 N(t,\F_q) \big) \label{asking Daniel} \\
&= \delta N(t,\F_q) \frac{\sin^2(\delta t)}{(\delta t)^2} \bigg|_{\delta^{-\frac 12}}^T - \delta \int_{\delta^{-\frac 12}}^T \Big(\frac{2\sin(\delta t)\cos(\delta t)}{\delta t^2} - \frac{2\sin^2(\delta t)}{\delta^2 t^3} \Big) N(t,\F_q) \,dt   \notag \\
&= -\delta \int_{\delta^{-\frac 12}}^T 2\Big(\frac{\sin(2\delta t)}{2\delta t^2} - \frac{\sin^2(\delta t)}{\delta^2 t^3} \Big) N(t,\F_q) \,dt +O\Big(\delta^{\frac 12} \Phi_q \log (q\delta^{-\frac 12})+ \frac{\Phi_q \log(qT)}{\delta T} \Big) \notag \\
&= -\frac{\Phi_q\delta}{\pi} \int_{\delta^{-\frac 12}}^T 2\Big(\frac{\sin(2\delta t)}{2\delta t^2} - \frac{\sin^2(\delta t)}{\delta^2 t^3} \Big) t \log (qt) \,dt   \notag \\
& \qquad{}+O\Big( \Phi_q \big(\delta\log(qT)\log T + \log(T\delta) \big) + E_q\log(T\delta) +\delta^{\frac 12 } \Phi_q \log (q\delta^{-1})+ \frac{\Phi_q \log(qT)}{\delta T}  \Big), \notag
\end{align}
using $\frac{\sin u}u \ll \min\{1,\frac1{|u|}\}$.
Finally, integrating by parts in the other direction, the main term in this expression equals
\begin{align*}
&-\frac{\Phi_q\delta}{\pi} t \log (qt) \frac{\sin^2(\delta t)}{(\delta t)^2} \bigg|_{\delta^{-\frac 12}}^T + \frac{\Phi_q\delta}{\pi}\int_{\delta^{-\frac 12}}^{T} \frac{\sin^2 (\delta t)}{(\delta t)^2} (\log (qt)+1)\,dt \\
&\qquad{}=\frac{\Phi_q}{\pi}\int_{0}^{\infty} \frac{\sin^2 u}{u^2} (\log (q\delta^{-1}u)+1)\,du+O\Big(\delta^{\frac 12}\Phi_q\log (q\delta^{-\frac12})+\frac{\Phi_q\log(qT)}{\delta T} \Big)  \\
&\qquad{}=\frac{\Phi_q}2 \log (q\delta^{-1})  +O\Big(\delta^{\frac 12}\Phi_q\log (q\delta^{-1})+\frac{\Phi_q\log(qT)}{\delta T} +\Phi_q \Big)
\end{align*} 
by the evaluation $\int_0^{\infty} \frac{\sin^2 u}{u^2} \,du = \frac \pi 2$ (which can be derived from equation~\eqref{sin2 fourier}).
\end{proof}

Now that we have evaluated some of the main terms in equation~\eqref{equation lemma explicit}, we can deduce a precise estimate of the values attained by $R_\delta(y)$ and $S_\delta(y)$ in Lemma~\ref{lemma explicit formula precise}.

\begin{lemma}
\label{lemma final expression for R}
Let $q\geq 1$ be an integer, and let~$\F_q$ be a set of characters modulo $q$ with the property~\eqref{equation property 1 F_q} such that GRH$(\chi)$ is true for all $\chi\in\F_q$.
Let $\delta>0$ be small enough (in absolute terms), $T\ge e\delta^{-1}$ and $M\geq 1$.
If $n\in \mathbb N$ satisfies $\| n \gamma_{\chi} \delta /2\pi \|< M^{-1} $ for each $ 0\leq \gamma_{\chi} \leq T$ with $\chi \in \F_q$, then $y=(n+1) \delta$ has the property that
\begin{multline*}
 R_{\delta}(y)= -\frac{e^{\frac y2} \Phi_q \log(q^2\delta^{-1})}{2} \\+ O\Big(e^{\frac y2}\Phi_q \Big(\frac{\log(qT)}{\delta T}+ \Big( \frac{E_q }{\Phi_q }+1\Big)\log( T\delta) + \frac{y\log q}{e^{\frac y2}}+ \frac{\log(qT)\log T}M  + \delta^{\frac 12} \log(qT)+\frac{\log q}{\log_2 q}\Big) \Big), 
 \end{multline*}
 where $E_q$ was defined in equation~\eqref{equation definition E_q}.
\end{lemma}

\begin{proof}
We first truncate part of the sum in Lemma \ref{lemma explicit formula precise}, in order to expand $\sin(\delta\gamma_\chi) $ and $\cos(\delta\gamma_\chi)$ into
 Taylor series. Doing so, we see that the main term in the estimate~\eqref{equation lemma explicit} equals
\begin{align} \notag & - e^{\frac y2}\sum_{\chi \in \F_q}\sum_{0 \leq \gamma_\chi \leq  \delta^{-\frac 12}}  \Big( \frac{2 \gamma_\chi\sin(\delta \gamma_\chi)\cos(\delta \gamma_\chi)}{\delta|\rho_\chi|^4}   +\Big(\frac 14-\gamma_\chi^2 \Big)\frac{\cos^2(\delta \gamma_\chi)}{|\rho_\chi|^4}  \Big) \\& \hspace{1cm}-e^{\frac y2} \delta \sum_{\chi \in \F_q}\sum_{0 \leq  \gamma_\chi  \leq T} \frac{2\gamma_\chi^4}{|\rho_\chi|^4}\frac{\sin^2(\delta \gamma_\chi)}{(\delta\gamma_\chi)^2} +O\Big( e^{\frac y2} \delta^{\frac 12} \Phi_q  \log (q\delta^{-1}) \Big)  \notag \\
\notag=& -e^{\frac y2} \delta \sum_{\chi \in \F_q}\sum_{0 \leq  \gamma_\chi  \leq T} \frac{2\gamma_\chi^4}{|\rho_\chi|^4}\frac{\sin^2(\delta \gamma_\chi)}{(\delta\gamma_\chi)^2} - e^{\frac y2}\sum_{\chi \in \F_q}\sum_{0 \leq \gamma_\chi  \leq  \delta^{-\frac 12}} \frac{ \frac 14+\gamma_\chi^2}{|\rho_\chi|^4} +O\Big(  e^{\frac y2}\delta^{\frac 12} \Phi_q \log (q\delta^{-1})\Big) \\
=&-e^{\frac y2}\delta \sum_{\chi \in \F_q}\sum_{\gamma_\chi \geq 0 } \frac{2\gamma_\chi^4}{|\rho_\chi|^4}\frac{\sin^2(\delta \gamma_\chi)}{(\delta\gamma_\chi)^2} - e^{\frac y2}\sum_{\chi \in \F_q}\sum_{\gamma_\chi \geq 0 } \frac{1}{|\rho_\chi|^2} +O\Big(  e^{\frac y2}\delta^{\frac 12} \Phi_q \log (q\delta^{-1}) + e^{\frac y2}\Phi_q\frac{ \log (qT)}{\delta T} \Big). \label{equation negativity}
\end{align}

The second double sum can be treated by using the functional equation (with Lemma~\ref{lemma multiplicities} accounting for possible zeros at $s=\frac12$), and applying~\cite[Lemma 3.5]{FiMa} and the Littlewood bound $L'(1,\chi) /L(1, \chi) \ll \log_2 q$; we obtain
\begin{align*}
\sum_{\chi \in \F_q} \sum_{\gamma_\chi \geq 0 } \frac{1}{|\rho_\chi|^2} &= \frac 12\sum_{\chi \in \F_q} \sum_{\gamma_\chi } \frac{1}{|\rho_\chi|^2} +O\Big(\frac{\Phi_q \log q}{\log_2 q}\Big)\\
&=\frac 12 \sum_{\chi \in \F_q} \log {q_\chi} +O\Big(\frac{\Phi_q \log q}{\log_2 q}\Big) =  
\frac{\Phi_q\log q}2+O\Big(\frac{\Phi_q \log q}{\log_2 q}+E_q\Big).
\end{align*}
Combining this evaluation with Lemma~\ref{lemma sum over zeros}, we find that the sum of the first two terms in equation~\eqref{equation negativity} equals
\begin{multline*}
  -e^{\frac y2}\frac{\Phi_q\log (q^2\delta^{-1})}2  \\+O\Big( e^{\frac y2} \Phi_q  \Big(\delta \log(qT)\log T+\delta^{\frac 12 }  \log (q\delta^{-1})+ \frac{\log(qT)}{\delta T}+\Big( \frac{E_q }{\Phi_q }+1\Big)\log(T\delta ) +\frac{\log q}{\log_2 q}\Big)  \Big).  \qedhere
  \end{multline*}
\end{proof}

Now that we know exactly how large $R_\delta(y)$ and $S_\delta(y)$ can be, it is time to understand more precisely the set of $y$ which are admissible. In particular, it is important for us to localize these values in terms of the modulus $q$. This will be done using a counting argument, inspired by the proof of~\cite[Lemma 2.4]{RubSar}.

\begin{lemma}
\label{lemma diophantine}
Let $\Lambda = \{\lambda_1,\ldots,\lambda_k\}$ be a set of real numbers. For any positive integers~$M$ and~$N$,
$$ \#\{ n\leq N \colon \|n \lambda \| \leq M^{-1} \text{ for all } \lambda \in \Lambda \}  \geq  \frac N{M^k}  -1. $$
\end{lemma}

\begin{proof}
Consider the integer multiples $n \textbf{v}$, with $1\le n\leq N$, of the vector $\textbf{v}=(\lambda_1,\ldots,\lambda_k) \in \mathbb R^k / \mathbb Z^k.$ If we divide $\mathbb R^k / \mathbb Z^k$ into $M^k$ cubes of side length $M^{-1}$, then one of these cubes will contain $s\ge N/M^k$ multiples of~$\textbf{v}$. If the integers producing these multiples are $m_1 < m_2< \dots < m_s$, then we have
 \[
 \{ n\leq N \colon \|n \lambda \| \leq M^{-1} \text{ for all } \lambda \in \Lambda \} \supset \{ m_2-m_1 , m_3-m_1, \dots m_s -m_1 \},
 \]
and the cardinality of the right-hand side is at least $N/M^k-1$.
\end{proof}

\begin{proposition}
\label{corollary GRH false hard}
Fix $\ep>0$ sufficiently small, and let $f,g\colon \mathbb N\rightarrow \mathbb R_{>0}$ be two functions such that $f$ is minorized by a large enough constant, and such that $(\eps^{-1}\log q)/\phi(q) \le g(q) \leq \log q$. If GRH is true, then for each sufficiently large~$q$ there exists $x_q$ satisfying 
$$\log_2 x_q \asymp_\ep \phi(q) f(q)g(q)\Big(1+\frac{\log (f(q))}{\log q}\Big)$$ with the property that
$$ \frac{G(x_q;q)}{x_q} \geq \Big(\frac 14-2\ep \Big) \frac{g(q)(\log(q^2f(q)))^2}{\log q} .  $$
Under the slightly weaker assumption that GRH$(\chi)$ is true for every nonprincipal character~$\chi$, the same statement holds with $G(x_q;q)$ replaced by $V_{\Lambda}(x_q;q)$.
\end{proposition}

\begin{proof}
We apply Lemma~\ref{lemma subset with big conductors} with $w(q)= (\log_2 q)^{-1}$. We deduce the existence of a set~$\mathcal G_q$ of characters modulo $q$ for which $|\mathcal G_q|\geq \phi(q)(1-K(\log_2 q)^{-2})$, where $K>0$ is an absolute constant, such that $\log q_\chi=\log q+O((\log_2 q)^2)$ for each $\chi\in \mathcal  G_q$. We can assume that all elements of $\mathcal  G_q$ are complex characters, since there are $\ll 2^{\omega(q)}\ll \sqrt q$ real characters modulo~$q$. Since $q_{\chi} = q_{\overline{\chi}}$, we can also assume that~$\mathcal  G_q$ has the property~\eqref{equation property 1 F_q}.
We extract a subset $\F_q \subset\mathcal  G_q$ of cardinality $$\Phi_q =2\Big\lfloor \frac{\phi(q) g(q)}{2\log q} (1-K(\log_2 q)^{-2})\Big\rfloor  $$ for which~\eqref{equation property 1 F_q} holds, where the right-hand side is at least~$2$ when~$q$ is sufficiently large.
By Lemma~\ref{lemma subset with big conductors}, the error term~\eqref{equation definition E_q} then satisfies the bound
$$ E_q \ll \Phi_q (\log_2 q)^2.$$

We now apply Lemma \ref{lemma diophantine} to the set $\Lambda = \{ 0\leq \gamma_{\chi} \leq T, \chi \in \F_q \}$. For $T$ large enough, it follows from equation~\eqref{equation von mangoldt} that the set $S$ of values of $n\leq N$ for which $ \|n\delta\gamma_{\chi} /2\pi\| < (2\pi M)^{-1}$ has at least $NM^{ (-1+o(1))\Phi_q T \log (qT) /2\pi  }-1$ elements. Taking $N= M^{\Phi_q T \log (qT)/\pi  }$, we obtain that $S\cap [N^{\frac 13},N ] \neq \varnothing$.
Then, 
Lemma \ref{lemma final expression for R} implies that for $y=\delta(n+1)$ with $n \in S$,
\begin{multline*}
S_{\delta}(y)= -\frac{e^{\frac y2} \Phi_q \log(q^2\delta^{-1})}{2} \\+ O\Big(e^{\frac y2}\Phi_q \Big(\frac{\log(qT)}{\delta T}+ (\log_2 q)^2\log(\delta T) +\frac{y\log q}{e^{\frac y2}}+ \frac{\log(qT)\log T}M  + \delta^{\frac 12} \log(qT)+\frac{\log q}{\log_2 q}\Big) \Big).  
 \end{multline*}
If $C=C(\ep)>0$ is large enough, then picking $T=C/\delta$, $M=C\log T$ and $0<\delta<C^{-2}$ will result, for $q$ and $y$ large enough, in the bound 
\begin{equation}
  S_{\delta}(y) \leq -e^{\frac y2} \Phi_q \log(q^2\delta^{-1})\Big( \frac {1-3\ep}4\Big)^{\frac 12}. 
  \label{equation Rdelta very negative}
\end{equation}
Now $S_{\delta}(y)$ is the average of the function $\sum_{\chi \in \F_q}( \theta(e^t,\chi)-\one_{\chi=\chi_0}e^t)$ over the short interval $e^t \in [e^{y-\delta},e^{y+\delta}]$, and hence this function itself has such a large negative value in that interval. In other words, there exists a value $x=e^y(1+O(\delta))$ such that
\begin{equation}  \label{omega for psi x chi}
\sum_{\chi \in \F_q} (\theta(x,\chi)-\one_{\chi=\chi_0} x)  \leq -x^{\frac 12} \Phi_q \log(q^2\delta^{-1})\Big( \frac {1-2\ep}4\Big)^{\frac 12},
\end{equation}
since $C$ is large enough in terms of $\eps$.
 Using positivity in~\eqref{equation parseval} and applying 
the Cauchy-Schwarz inequality, we obtain that 
\begin{align} \notag
 G(x;q) &\geq  \frac 1{\phi(q)}\sum_{\chi  \in \F_q} |\theta(x,\chi)-\one_{\chi=\chi_0}x|^2 \geq \frac 1{\phi(q)\Phi_q}\Big|\sum_{\chi  \in \F_q}(\theta(x,\chi)-\one_{\chi=\chi_0}x)\Big|^2 \\&\geq x  \frac{\Phi_q}{\phi(q)}(\log (q^2 \delta^{-1}))^2\Big( \frac {1-2\ep}4\Big). \label{CS on H GRH}
\end{align}
Since $y=\delta(n+1)$ with $n\in [N^{\frac 13},N]$, it follows that the associated $x$ satisfies
\begin{equation} \label{use this for fixed chi}
\log_2 x \asymp_{\ep} \delta^{-1} \log_2(\delta^{-1}) \Phi_q \log (q\delta^{-1}).
\end{equation}
The result follows from 
taking $\delta = f(q)^{-1}\log_2 f(q)$.

The proof is identical for $V_{\Lambda}(x;q)$. 
\end{proof}

We are now ready to prove our main technical theorem, at which point we will be able to deduce Theorems~\ref{theorem main}, \ref{theorem main precise}, \ref{theorem fixed q}, and~\ref{theorem average over q}.

\begin{theorem}
\label{th:main}
Assume GRH, and fix $\ep>0$ small enough.
\begin{enumerate}
\item If $h(x)$ is an increasing function satisfying
\begin{equation}
\ep \frac{\log_2 x}{ \log_3 x} \leq h(x) \leq (\log x)^{\frac{\ep^2}3}
\label{equation condition on h GRH theorem}
\end{equation}
for all $x\geq e^3$, then for all moduli~$q$ there exist associated values~$x_q$ satisfying 
$ h(\exp((\log x_q)^{c_1\ep^{-1}})) \leq \phi(q) \leq  h(\exp((\log x_q)^{c_2\ep^{-1}})) $ such that
$$ G(x_q;q) \gg_\ep   x_q\log q \cdot  \frac{ \log_2 x_q }{\phi(q)} . $$
\item If $h(x)$ is a function with the property~\eqref{nice function} and satisfying
\begin{equation}
h(x)  \leq  \ep \frac{\log_2 x}{\log_3 x}
\label{equation stronger condition on h}
\end{equation}
for all $x\geq e^3$, then for all sufficiently large moduli~$q$ there exist associated values~$x_q$ satisfying $\phi(q)\asymp_\ep h(x_q)$ such that
$$ G(x_q;q) \geq \Big(\frac 14-\ep\Big) x_q (\log q + \log_3 x_q)^2. $$
\end{enumerate}
These results hold with $V_{\Lambda}(x_q;q)$ in place of $G(x_q;q)$, under the weaker assumption that GRH$(\chi)$ is true for every nonprincipal character~$\chi$.
\end{theorem}

\begin{proof}
Under the condition~\eqref{equation condition on h GRH theorem}, we apply Proposition~\ref{corollary GRH false hard} with $f(q)$ equal to a sufficiently large absolute constant, and with $g(q)=\ep\log_2 (h^{-1}(\phi(q)))/(2\phi(q)) $. 
Note that the inequality $h(\exp( q^{2\ep^{-1}\phi(q)})) \geq \phi(q)$ holds for~$q$ large enough, and thus $g(q) \leq \log q.$ Moreover, $ h(\exp(q^{2\ep^{-2}})) \leq q^{ \frac 23} \leq \phi(q)$, and hence $g(q) \geq  (\eps^{-1}\log q) / \phi(q)$. Therefore, the hypotheses of Proposition~\ref{corollary GRH false hard} are satisfied. We deduce the existence of a sequence $\{x_q\}_{q\geq q_0}$ such that $\log_2 x_q \asymp \phi(q) g(q)$ and
$$ \frac{G(x_q;q)}{x_q} \gg  g(q)\log q \asymp \frac{\log_2 x_q \cdot \log q }{\phi(q)},  $$
establishing part~(a).

On the other hand, under the conditions~\eqref{nice function} and~\eqref{equation stronger condition on h}, we make the choices $g(q)=\log q$ and 
$$f(q): =\begin{cases}
\displaystyle \frac{\log_2 (h^{-1}(\phi(q)))}{\phi(q)\log q_{\mathstrut}}, & \text{ if } \log_2 (h^{-1}(\phi(q))) \leq q^2, \\
\displaystyle  \frac{\log_2 (h^{-1}(\phi(q))) }{\phi(q)\log_3 (h^{-1}(\phi(q)))^{\mathstrut}}, & \text{ otherwise}.
\end{cases}$$
The condition~\eqref{equation stronger condition on h} ensures that $f(q)$ is minorized by a large enough positive constant when~$q$ is sufficiently large. Moreover, one can check that 
$$\phi(q) f(q) \log q \cdot \Big(1+\frac{\log (f(q))}{\log q}\Big) \asymp \log_2 (h^{-1}(\phi(q))).$$
Hence, Proposition~\ref{corollary GRH false hard} (applied with $\frac\ep4$ in place of~$2\ep$) yields a real number $x_q$ satisfying $\log_2 x_q \asymp \log_2 (h^{-1}(\phi(q)))$ with the property that if $\log_2 (h^{-1}(\phi(q))) > q^2$,
\begin{align*}
\frac{G(x_q;q)}{x_q} &\ge \Big( \frac14-\frac\ep2 \Big) (\log(q^2 f(q)))^2 \\
&\ge \Big( \frac14-\frac\ep2 \Big) \Big(\log \Big( \frac{q \log_2 (h^{-1}(\phi(q)))}{\log_3 (h^{-1}(\phi(q)))} \Big)\Big)^2 \ge \Big( \frac14-\ep \Big) \big(\log \big( q \log_2 (h^{-1}(\phi(q))) \big)\big)^2,
\end{align*}
when~$q$ is sufficiently large, and similarly when $\log_2 (h^{-1}(\phi(q))) \leq q^2$.
Moreover, the estimate $\log_2 x_q \asymp \log_2 (h^{-1}(\phi(q)))$ combined with the property~\eqref{nice function} implies that $\phi(q)\asymp h(x_q)$, establishing part~(b).
\end{proof}

\begin{proof}[Proof of Theorems~\ref{theorem main} and~\ref{theorem main precise}]
We prove Theorem~\ref{theorem main precise} which implies Theorem~\ref{theorem main}. If we assume that GRH is false, then the desired result for $G(x;q)$ follows from Proposition~\ref{prop:GRHfalse}. On the other hand, if we assume that GRH holds, then the desired result follows from applying Theorem~\ref{th:main}, which holds for all moduli~$q$, and then restricting to the positive proportion of moduli~$q$ that satisfy $\phi(q) \ge \frac12q$, say. (The constant $\frac12$ is unimportant here; any constant less than $1$ suffices, since we know~\cite[Theorem 1, \S 8]{Sch2} (see also~\cite[Section~5]{Sch}) that the limiting distribution function $\phi(q)/q$ is strictly increasing on~$(0,1)$.)

The proof is similar for $V_{\Lambda}(x;q)$, and the Riemann hypothesis for principal characters~$\chi_0$ is never needed (see the formulas~\eqref{equation parseval}).
\end{proof}

\begin{proof}[Proof of Theorem~\ref{theorem fixed q}]
If GRH$(\chi)$ is false, then the desired result for $\theta(x,\chi)$ follows from Proposition~\ref{proposition GRH false single chi}. If GRH$(\chi)$ is true, then we argue analogously to the proof of Proposition~\ref{corollary GRH false hard}.

Take $\F_q=\{\chi,\overline{\chi}\}$ in Lemma~\ref{lemma final expression for R}, as well as  $T=C/\delta$, $M=C\log T$ and $\delta< C^{-2}$ with $C$ large enough. Take moreover $N= M^{   T \log (qT) /\pi }$ in Lemma~\ref{lemma diophantine}. Hence, there exists $ n\in [N^{\frac 13},N]$ such that $y=(n+1)\delta$ has the property that 
$$
S_{\delta}(y)\leq  - \frac 12  e^{\frac y2}  \log(\delta^{-1}) .
$$
Since $ \log_2 x= \log y \asymp_q \delta^{-1} \log(\delta^{-1})\log_2(\delta^{-1}), $ we have that
$ e^{\frac y2}  \log(\delta^{-1})\asymp_q  x^{\frac 12} \log_3 x$ and the lower bound~\eqref{equation omega result fixed chi} follows. The proof of the lower bound~\eqref{equation omega result fixed q} is similar, this time taking~$\F_q$ to be the set of all characters modulo~$q$ and applying equation~\eqref{CS on H GRH}.  
\end{proof}

\begin{proof}[Proof of Theorem~\ref{theorem average over q}]
If $Q(x)$ is bounded, and thus is eventually constant $Q(x)=Q_0$, then the result follows from the bound
$$ \sum_{ Q_0<q \leq 2Q_0 } G(x;q) \geq  G(x;2Q_0) $$
and Theorem~\ref{theorem fixed q}. 

We now assume that $Q(x)$ tends to infinity. If GRH is false, we let $\chi_e \bmod q_e$ be a primitive character for which $L(s,\chi_e)$ has a non-trivial zero off the critical line. Then, for $x$ large enough the interval $(Q(x),2Q(x)]$ will contain a multiple $q_j$ of $q_e$. Hence, if $\chi_j \bmod q_j$ is the character induced by $\chi_e$, we have that
$$ \sum_{ Q(x)<q \leq 2Q(x) } G(x;q) \geq  \frac{|\theta(x,\chi_j)-\one_{\chi_j=\chi_0} |^2}{\phi(q)}, $$
and the rest of the proof proceeds as in the proof of Proposition~\ref{proposition Hooley implies GRH}. 

If GRH is true, then we argue as in the proof of Proposition~\ref{corollary GRH false hard}. We apply Lemma~\ref{lemma diophantine} to the set $\Lambda = \{ 0\leq \gamma_{\chi} \leq T\colon \chi \bmod q,\, Q< q\leq 2Q \}$. Taking $N= M^{ 2 Q^2 T \log (QT)/\pi  }$, we see that the set $S$ of values of $n\leq N$ for which $ \|n\delta_Q\gamma_{\chi} /2\pi\| < (2\pi M)^{-1}$ has at least one element exceeding~$N^{\frac 19}$. 
Then we set $T=C/\delta_Q$, $M=C\log T$ and $\delta_Q\leq C^{-2}$ with $C$ large enough in
Lemma~\ref{lemma final expression for R}, and obtain that for $y=\delta_Q(n+1)$ with $n \in S$,
$$
S_{\delta_Q}(y)\leq  -\Big( \frac 12-\ep\Big)  e^{\frac y2} \phi(q) \log(q^2\delta_Q^{-1}) .
$$
Hence, as in~\eqref{CS on H GRH}, for each large enough $Q$ there exists $x_Q$ such that
$$ \frac 1Q \sum_{ Q<q\leq 2Q} G(x_Q;q) \gg  \frac {x_Q}Q \sum_{ Q<q\leq 2Q}  (\log (q^2\delta_Q^{-1}))^2  \gg x_Q (\log (Q^2\delta_Q^{-1}))^2,  $$
and for which
$$ \log_2 x_Q \asymp Q^2 \delta_Q^{-1} \log(Q\delta_Q^{-1}). $$
The rest of the proof proceeds as in the proof of Theorem~\ref{th:main}.
\end{proof}

\section*{Acknowledgements}

The work of the first author was supported at the University of Ottawa by an NSERC discovery grant and at the Institut Math\'ematique de Jussieu by a postdoctoral fellowship from the Fondation Sciences Math\'ematiques de Paris, and was initiated at the University of Michigan. The second author was supported in part by a Natural Sciences and Engineering Research Council of Canada Discovery Grant. We would like to thank R\'egis de la Bret\`eche for helpful discussions and for his comments.

\end{document}